\documentclass[12pt]{amsart}

\usepackage{amssymb, bm,color, ytableau}
\newtheorem{theorem}{Theorem}[section]
\newtheorem{proposition}[theorem]{Proposition}
\newtheorem{lemma}[theorem]{Lemma}
\newtheorem{corollary}[theorem]{Corollary}

\theoremstyle{definition}

\newtheorem{example}[theorem]{Example}
\newtheorem{problem}[theorem]{Problem}

\newtheorem{remark}[theorem]{Remark}

\numberwithin{equation}{section}

\makeatletter

\@namedef{subjclassname@2020}{%

\textup{2020} Mathematics Subject Classification}

\makeatother

\begin{document}

\title[Gr\"obner bases of radical Li-Li type ideals]{Gr\"obner bases of radical Li-Li type ideals associated with partitions}

\author{Xin Ren}
\address{
Xin Ren,
Department of Mathematics, 
Kansai University, Suita-shi, Osaka 564-8680, Japan.
}
\email{k641241@kansai-u.ac.jp}

\author{Kohji Yanagawa}
\address{
Kohji Yanagawa,
Department of Mathematics, 
Kansai University, Suita-shi, Osaka 564-8680,  Japan.
}
\email{yanagawa@kansai-u.ac.jp}

\thanks{The second authors is partially supported by JSPS KAKENHI Grant Number 22K03258.}
\date{\today}
\keywords{ Gr\"obner bases, Specht polynomial, Specht ideal, Li-Li ideal}
\subjclass[2020]{13P10, 05E40}

\newcommand{\add}[1]{\ensuremath{\langle{#1}\rangle}}
\newcommand{\tab}[1]{\ensuremath{{\rm Tab}({#1})}}
\newcommand{\wtab}[1]{\ensuremath{\widehat{\rm Tab}({#1})}}
\newcommand{\stab}[1]{\ensuremath{{\rm STab}({#1})}}
\newcommand{\ini}[1]{\ensuremath{{\rm in}({#1})}}
\newcommand{\Ini}{\operatorname{in}}
\newcommand{\Fc}{{\mathcal F}}
\newcommand{\Yc}{{\mathcal Y}}
\newcommand{\Ib}{{\mathbf I}}
\newcommand{\Vb}{{\mathbf V}}
\newcommand{\fS}{{\mathfrak S}}
\newcommand{\cS}{{\mathcal S}}
\newcommand{\cmpl}{{\sf c}}
\newcommand{\NN}{{\mathbb N}}
\newcommand{\ZZ}{{\mathbb Z}}
\newcommand{\wT}{{\widetilde{T}}}
\newcommand{\wlam}{{\widetilde{\lambda}}}
\newcommand{\sh}{\operatorname{sh}}
\newcommand{\Tc}{{\mathcal T}}
\newcommand{\wTc}{{\widetilde{\Tc}}}
\def\<{\langle}
\def\>{\rangle}
\def\too{\longrightarrow}

\maketitle

\begin{abstract}
For a partition $\lambda$ of $n$, the {\it Specht ideal} $I_\lambda \subset K[x_1, \ldots, x_n]$ is the ideal generated by all Specht polynomials of shape $\lambda$.   
In their unpublished manuscript, Haiman and Woo showed that $I_\lambda$ is a radical ideal, and gave its universal Gr\"obner basis (Murai et al. published a quick proof of this result). 
On the other hand, an old paper of Li and Li studied analogous ideals, while their ideals are not always radical.  
The present paper introduces a class of ideals generalizing both Specht ideals and {\it radical} Li-Li ideals, and studies their radicalness and Gr\"obner bases. 
\end{abstract}

\section{Introduction}
Let $S=K[x_1, \ldots, x_n]$ be a polynomial ring over an infinite field $K$. For a subset $A=\{a_1, a_2, \ldots, a_m\}$ of $[n]:=\{1,2, \ldots, n\}$, let 
$$\Delta(A):=\prod_{1 \le i <j \le m}(x_{a_i} -x_{a_j}) \in S$$
be the difference product. 
For a sequence of subsets $\Yc=(Y_1, Y_2, \ldots, Y_{k-1})$ with $[n]\supset Y_1 \supset Y_2 \supset \cdots \supset Y_{k-1}$, Li and Li \cite{LL} studied 
the ideal 
\begin{equation}\label{I_Y}
I_\Yc:=\left( \, \prod_{i=1}^{k-1} \Delta(X_i) \, \middle | \,   X_i \supset Y_i \ \text{for all} \ i,  \  \bigcup_{i=1}^{k-1} X_i=[n] \, \right )
\end{equation}
of $S$ (more precisely, the polynomial ring in \cite{LL} is $\ZZ[x_1, \ldots, x_n]$). 
Among other things, they showed the following. 

\begin{theorem}[{c.f. Li-Li \cite[Theorem~2]{LL}}]
With the above notation, $I_\Yc$ is a radical ideal if and only if $\# Y_2 \le 1$. 
\end{theorem}


A {\it partition} of a positive integer $n$ is a non-increasing sequence of positive integers $\lambda=(\lambda_1,\ldots,\lambda_p)$ with $\lambda_1+ \cdots+\lambda_p=n$. Let $P_n$ be the set of all partitions of $n$. 
A partition $\lambda$ is frequently represented by its Young diagram. For example, $(4,2,1)$ is represented as $\ytableausetup{mathmode, boxsize=0.5em}\ydiagram {4,2,1}$. A {\it (Young) tableau} of shape $\lambda  \in P_n$ is a bijective filling of the squares of the Young diagram of  $\lambda$ by the integers in $[n]$. 
For example, 
$$
\ytableausetup{mathmode, boxsize=1.2em}
\begin{ytableau}
4 & 3 & 1&7   \\
5 & 2    \\
6  \\  
\end{ytableau}
$$
is a tableau of shape $(4,2,1)$. Let $\tab \lambda$ be the set of all tableaux of shape $\lambda$. 
Recall that the Specht polynomial $f_T$ of $T \in \tab \lambda$ is $\prod_{j=1}^{\lambda_1} \Delta(T(j))$, where $T(j)$ is the set of the entries of the $j$-th column of $T$ (here the entry in the $i$-th row is the $i$-th element of $T(j)$). 
For example, if $T$ is the above tableau,   
then $f_T=(x_4-x_5)(x_4-x_6)(x_5-x_6)(x_3-x_2)$.  

We call the ideal 
$$I_\lambda:=(f_T \mid T \in \tab{\lambda}) \subset S$$
the {\it Specht ideal} of $\lambda$. These ideals have been studied from several points of view (and under several names and characterizations), see for example, \cite{BGS, MW,MRV, SY}. 
The following is an unpublished result of Haiman and Woo (\cite{HW}), to which Murai,  Ohsugi and the second author (\cite{MOY}) published a quick proof. 

\begin{theorem}[{Haiman-Woo \cite{HW}, see also \cite{MOY}}]\label{HWMOY}
If $\Fc \subset P_n$ is a lower filter with respect to the dominance order  $\unlhd$, then  $I_\Fc:=\sum_{\lambda \in \Fc} I_\lambda$ is a radical ideal, for which $\{ f_T \mid T \in  \tab{\mu}, \mu \in \Fc  \}$ forms a universal Gr\"obner basis (i.e., a Gr\"obner basis with respect to all monomial orders). In particular, $I_\lambda$ is a radical ideal, for which $\{ f_T \mid T \in  \tab{\mu}, \mu \unlhd \lambda  \}$ forms a universal Gr\"obner basis. 
\end{theorem}

Let us explain why the second assertion follows form the first. Since  $\lambda \unrhd \mu$ for $\lambda, \mu \in P_n$ implies $I_\lambda \supset I_\mu$ (c.f. Lemma~\ref{inclusion}), we have  $I_\lambda=I_\Fc$ for the lower filter $\Fc:=\{ \mu \in P_n \mid \mu \unlhd \lambda \}$. 

The Li-Li ideals $I_\Yc$ and the Specht ideals $I_\lambda$ share common examples. In fact,  for  $\Yc=(Y_1, Y_2, \ldots, Y_{k-1})$ with $\# Y_1 \le 1$, $Y_2=\cdots =Y_{k-1}=\emptyset$ and $\lambda=(\lambda_1, \ldots, \lambda_p) \in P_n$ with  $\lambda_1=\cdots =\lambda_{p-1}=k-1$, we have  $I_\Yc=I_\lambda$ by \cite[Corollary~3.2]{LL}.

In this paper, we study a common generalization of the {\it radical} Li-Li ideals and the Specht ideals, for which {\it almost} direct analogs of Theorem~\ref{HWMOY}  hold. 
For example, in Sections 2 and 3, we take a positive integer $l$, and a partition $\lambda \in P_{n+l-1}$ with $\lambda_1 \ge l$, and consider tableaux like 
\begin{equation}\label{Tab(l)}
\ytableausetup{mathmode, boxsize=1.2em}
\begin{ytableau}
1& 1& 1& 1& 2 & 3   \\
4 & 5   & 8  \\
6 & 7\\
\end{ytableau}
\end{equation}
($l=4$ in this case). Using these tableau, we define the ideal $I_{l, \lambda}$. 

The symmetric group $\cS_{n-1}$ of the set $\{2, \ldots, n\}$ still acts on $I_{l, \lambda}$, so our ideals have representation theoretic interest. The following are other motivations of the present paper. 
\begin{itemize}
\item[(1)] Recently, the defining ideals of subspace arrangements have been intensely studied (c.f. \cite{BPS, CT,S}). Our $I_{l, \lambda}$ and its generalization 
$\sqrt{I_{l,m,\lambda}}$ introduced in Section 4 give new classes of these ideals. Note that $I_{l,m,\lambda}$ is not a radical ideal in general, while Corollary~\ref{I_{l,m,lambda}} gives the generators of  its radical explicitly. 
\item[(2)] A universal Gr\"obner basis is very important, since it is closely related to the Gr\"obner fan.   While we can use a computer for explicit examples,  it is extremely difficult to construct universal Gr\"obner bases for some (infinite) family of ideals.  Theorem~\ref{HWMOY} gives universal Gr\"obner bases of Specht ideals $I_\lambda$. However, since $I_\lambda$ are symmetric, this  case is exceptional. So it must be very interesting, if the Gr\"obner bases of non-symmetric ideals $I_{l, \lambda}$ given in Theorem~\ref{main1} are universal. Corollary~\ref{smallest & largest} is an affirmative evidence. 
\item[(3)] One of the motivations of the paper \cite{LL} of Li and Li is an application to graph theory (see \cite{deL} for further connection to Gr\"obner bases theory). We expect that the present paper gives a new inslight to this direction.  
\end{itemize}

In the present paper, for the convention and notation of the Gr\"{o}bner bases theory, we basically follow \cite[Chapter 1]{HHO}.

\section{A generalization of the case $\#Y_1=\cdots =\#Y_l=1$}
We keep the same notation as Introduction, and fix a positive integer $l$.   
For $\lambda \in [P_{n+l-1}]_{\ge l}:=\{ \lambda \in P_{n+l-1} \mid \lambda_1 \ge l \}$, we consider  a bijective filling of the squares of the Young diagram of  $\lambda$
by the multiset $\{\overbrace{1, \ldots, 1}^{\text{$l$ copies}},2, \ldots, n\}$ such that no two copies of 1 are contained in the same column.  Let $\tab {l, \lambda}$ be the set of such tableaux. 
For example, the tableau \eqref{Tab(l)} above is an element of $\tab{4, \lambda}$ for  $\lambda=(6,3,2)$ (moreover, this is a {\it standard} tableau defined below).  
The Specht polynomial $f_T$ of $T \in \tab{l, \lambda}$ is defined by the same way as in the classical case. For example, 
if $T$ is the one in  \eqref{Tab(l)}, then $f_T=(x_1-x_4)(x_1-x_6)(x_4-x_6)(x_1-x_5)(x_1-x_7)(x_5-x_7)(x_1-x_8)$. 
For $\lambda \in [P_{n+l-1}]_{\ge l}$, consider the ideal 
$$I_{l, \lambda}:=(f_T \mid T \in \tab{l, \lambda})$$
of $S$. Clearly, $\tab{1, \lambda}=\tab{\lambda}$ and $I_{1,\lambda}=I_\lambda$.

For $\lambda=(\lambda_1,\dots,\lambda_p), \mu =(\mu_1,\dots,\mu_q) \in P_m$,
we write $\lambda \unrhd \mu$ if $\lambda$ is equal to or larger than $\mu$ with respect to the {\it dominance order},
that is,
$$\lambda_1+ \cdots+\lambda_i \geq \mu_1+ \cdots + \mu_i \ \ \ \mbox{ for }i =1,2,\dots,\min\{p, q\}.$$
In what follows, we regard $[P_{n+l-1}]_{\ge l}$ as a subposet of $P_{n+l-1}$. 

For $\lambda \in P_m$ and $j$ with $1 \le j \le \lambda_1$, let $\lambda^\perp_j$ be the length of the $j$-th column of the Young diagram of $\lambda$.   
Then $\lambda^\perp=(\lambda^\perp_1, \lambda^\perp_2, \ldots)$ is a partition of $m$ again.  It is a classical result that $\lambda \unrhd  \mu$ if and only if $ \lambda^\perp \unlhd  \mu^\perp$.

\begin{remark}\label{cover}
By \cite[Proposition~2.3]{B}, if $\lambda$ covers $\mu$ (i.e., $\lambda \rhd  \mu$, and there is no other partition between them), then there are two integers $i, i'$ with $i< i'$ such that $\mu_i=\lambda_i-1$, $\mu_{i'}=\lambda_{i'}+1$, and $\mu_k=\lambda_k$ for all $k\ne i, i'$, equivalently, there are two integers $j, j'$ with $j< j'$ such that $\mu^\perp_j=\lambda^\perp_j+1$, $\mu^\perp_{j'}=\lambda^\perp_{j'}-1$, and $\mu^\perp_k=\lambda^\perp_k$ for all $k\ne j, j'$. Clearly, $\mu_j^\perp \ge \mu_{j'}^\perp +2$ in this case. 
Here, we allow the case $i'$ is larger than the length $p$ of $\lambda$, where we set $\lambda_{i'}=0$. Similarly, the case $\mu^\perp_{j'}=0$ might occur. 
\end{remark}

\begin{remark}
By a similar argument to the proof of Lemma~\ref{stab}, to generate $I_{l,\lambda}$, it suffices to use $T \in \tab{l, \lambda}$ such that the left most $l$ squares in the first row are filled by 1.  So, in manner of \eqref{I_Y}, the ideal $I_{l, \lambda}$ can be represented as follows. 
$$I_{l, \lambda}=\left( \, \prod_{i=1}^{\lambda_1} \Delta(X_i) \, \middle | \,   1 \in X_i  \ \text{for} \ 1 \le i \le l,   \   \# X_i=\lambda_i^\perp \ \text{for all} \ i,  \  \bigcup_{i=1}^{\lambda_1} X_i=[n] \, \right )$$
\end{remark}

\noindent{\bf Convention.} Throughout this paper, when we consider the Gr\"obner bases, we use  the lexicographic order with $x_1 < \cdots < x_n$ unless otherwise specified (see Lemma~\ref{product of linear forms} below, which states that only  the order among the variables $x_1, \ldots, x_n$ matters for our Gr\" obner bases), and the initial monomial $\operatorname{in}_<(f)$ of $0 \ne f \in S$ will be simply denoted by $\ini{f}$. 

\medskip

For $T \in \tab{l,\lambda}$, recall that $T(j)$ is the set of the entries of the $j$-th column of $T$. If $\sigma$ is a permutation on $T(j)$, we have $f_{\sigma T} ={\rm sgn}(\sigma)f_T$ for each $j$.  
In this sense, to consider $f_T$, we may assume that $T$ is {\it column standard}, that is, all columns are increasing from top to bottom (in particular, all 1's appear in the 1st row).  

If $T$ is column standard and the number $i$ is in the $d_i$-th row of $T$, we have 
\begin{equation}\label{in(f)}
\ini {f_T} = \prod_{i=1}^n x_i^{d_i-1}
\end{equation}
 (recall our convention on the monomial order). 

If a column standard tableau $T \in \tab{l, \lambda}$ is also row semi-standard (i.e., all rows are non-decreasing from left to right), we say $T$ is {\it standard}. 
Let $\stab{l, \lambda}$ be the set of standard tableaux in $\tab{l, \lambda}$. We simply denote $\stab{1, \lambda}$ by $\stab{\lambda}$. The next result is very classical when $l=1$.   

\begin{lemma}\label{stab} 
For $\lambda \in [P_{n+l-1}]_{\ge l}$, $\{ f_T \mid T\in \stab{l, \lambda} \}$ forms a basis of the vector space $V$ spanned by  $\{ f_T \mid T\in \tab{l, \lambda} \}$. Hence $\{ f_T \mid T\in \stab{l, \lambda} \}$ is a minimal system of generators of $I_{l, \lambda}$. 
\end{lemma}

\begin{proof}
In the classical case (i.e., when $l=1$), we can rewrite $f_T$ for $T \in \tab{\lambda}$ as a linear combination of $f_{T_i}$'s for $T_i \in \stab{\lambda}$ repeatedly using the relations given by {\it Garnir elements} (see \cite[\S 2.6]{Sa}). Such a relation concerns the $j$-th and the $(j+1)$-st  columns of $T$.   The classical argument directly works in our case unless both of these columns contain 1.  
So we assume that both columns have 1. Since $f_T=\prod_{j=1}^{\lambda_1} \Delta(T(j))$, we can  concentrate on the $j$-th and $(j+1)$-st columns of $T$, and may assume that $T$ consists of two columns (i.e., $\lambda$ is of the form $(2,\lambda_2, \ldots, \lambda_p) \in P_{n+2-1}=P_{n+1}$) and $l=2$. 
Set $\wlam :=(\lambda_2, \ldots, \lambda_p) \in P_{n-1}$.  
Removing the first row from $T \in \tab{2, \lambda}$, we have $\wT \in \tab{\wlam}$ (the set of the entries of $\wT$ is $\{2, \ldots, n\}$). The converse operation $\tab{\wlam} \ni \wT \longmapsto T \in \tab{2, \lambda}$ also makes sense. 
Clearly, $f_T=(\prod_{i=2}^n(x_1-x_i)) \cdot f_{\wT}$. Multiplying $\prod_{i=2}^n(x_1-x_i)$  to both sides of a Garnir relation $f_\wT =\sum_{i=1}^k \pm f_{\wT_i}$ ($\wT, \wT_i \in \tab{\wlam}$), we have the relation  
$f_T =\sum_{i=1}^k \pm f_{T_i}$ ($T, T_i \in \tab{2, \lambda}$). As in the classical case, $T_i$ need not to be standard, but is closer to standard than  $T$. 
Using these relations, the argument in \cite[\S 2.6]{Sa} is applicable to our case, and we can show that  $\{ f_T \mid T\in \stab{l, \lambda} \}$ spans $V$. 

As we have seen in \eqref{in(f)}, 
$\ini{f_T} \ne \ini{f_{T'}}$ holds for distinct $T,T'\in \stab{l, \lambda}$. So  $\{ f_T \mid T\in \stab{l, \lambda} \}$  is linearly independent. 
\end{proof}

For $\bm a =(a_1, \ldots, a_n) \in K^n$, there are  distinct  $\alpha_1, \ldots, \alpha_p \in K$ with $\{\alpha_1, \ldots, \alpha_p\} =\{a_1, \ldots, a_n\}$ as sets. Now we can define the  partition $\mu =(\mu_1,\dots,\mu_p) \in P_n$ such that $\alpha_i$ appears $\mu_i$ times in  $(a_1, \ldots, a_n)$ for each $i$.  
This partition $\mu$  will be denoted by $\Lambda(\bm a)$.
For example, $\Lambda((1,0,2,1,2,2))=(3,2,1)$. 

For $\bm a \in K^n$, set 
$\bm a_{(l)}:=(\overbrace{a_1, \ldots, a_1}^{\text{$l$ copies}}, a_2, \ldots, a_n) \in K^{n+l-1}$ and $\Lambda_l(\bm a):=\Lambda(\bm a_{(l)}) \in [P_{n+l-1}]_{\ge l}$. 
For example, if $\bm a = (1,0,2,1,2,2)$, then $\bm a_{(3)}=(1,1,1, 0,2,1,2,2)$ and $\Lambda_3 (\bm a) =(4,3,1)$. 
  When $l=1$, the following result is classical. 

\begin{lemma}[{c.f. \cite[Lemma~2.1.]{MOY}}]
\label{f_T(a)=0}
Let $\lambda \in [P_{n+l-1}]_{\ge l}$ and $T \in \tab{l, \lambda}$.  
For $\bm a \in K^n$ with $\Lambda_l(\bm a) \not \not \! \unlhd \lambda$, we have $f_T(\bm a)=0$.
\end{lemma}

\begin{proof}
For $\bm a=(a_1, \ldots, a_n) \in K^n$, replacing $i$ with $a_i$ for each $i$ in $T$, we have a tableau $T(\bm a)$, whose entries are elements in $K$.    It is easy to see that $f(\bm a) \ne 0$ if and only if the entries in each column of $T(\bm a)$ are all distinct. So the assertion follows from the same argument as \cite[Lemma~2.1]{MOY}. 
\end{proof}

\begin{lemma}[{c.f. \cite[Theorem~1.1]{MRV}}]
\label{inclusion}
For $\lambda, \mu \in [P_{n+l-1}]_{\ge l}$ with $\lambda \unrhd \mu$, we have $I_{l, \lambda} \supset I_{l, \mu}$. 
\end{lemma}

\begin{proof}
The proof is essentially same as the classical case, while we have to care about one point. 
First, we will recall a basic property of difference products. 
For subsets $A=\{a_1, a_2, \ldots, a_k\}$ and  $B=\{b_1, b_2, \ldots, b_{k'} \}$ of $[n]$ with $k \ge k'+2$, 
we have  
\begin{equation}\label{diff product}
\Delta(A)\cdot \Delta(B)=\sum_{k-k' \le i \le k} (-1)^{i-k+k'}\left[\Delta(A\setminus \{a_i\}) \cdot \Delta(B \cup \{a_i\}) \cdot \prod_{1 \le i' < k-k'}(x_{a_{i'}}-x_{a_i}) \right]
\end{equation}
by \cite[Proposition~3.1]{LL}, where we regard $a_i$ as the last element of $B \cup \{a_i\}$.  

Let us start with the main body of the proof. To prove the assertion, we may assume that $\lambda$ covers $\mu$. By Remark~\ref{cover}, there are $j,j'$ with $j<j'$ such that $\mu^\perp_j=\lambda^\perp_j+1$, $\mu^\perp_{j'}=\lambda^\perp_{j'}-1$, and $\mu^\perp_i=\lambda^\perp_i$ for all $i \ne j, j'$. Take $T \in \tab{l, \mu}$, and let $A=\{a_1, \ldots, a_k\}$ (resp. $B=\{b_1, \ldots, b_{k'} \}$) be the set of the contents of the $j$-th (resp. $j'$-th) column of $T$. For $i$ with $k-k' \le i \le k$, consider the tableau $T_i$ whose $j$-th (resp. $j'$-th) column consists of the elements of $A\setminus \{a_i\}$ (resp. $B \cup \{a_i\}$) and the other columns are same as those of $T$. Since $a_i \ge 2$ for $i \ge 2$, we have $T_i \in \tab{l, \lambda}$.  By \eqref{diff product}, we have 
\begin{equation}\label{inclusion expansion}
f_T=\sum_{k-k' \le i \le k} (-1)^{i-k+k'}\left[f_{T_i} \cdot \prod_{1 \le i' < k'-k}(x_{a_{i'}}-x_{a_i}) \right] \in I_{l, \lambda},
\end{equation}
and it means that $I_{l, \lambda} \supset I_{l, \mu}$. 
\end{proof}

We say that $\Fc \subset [P_{n+l-1}]_{\ge l}$ is a {\it lower (resp. upper) filter} if $\lambda \in \Fc$, $\mu \in  [P_{n+l-1}]_{\ge l}$ and $\mu \unlhd \lambda$ (resp. $\mu \unrhd \lambda$) imply $\mu \in \Fc$. 
For a {\it lower} filter $\Fc \subset [P_{n+l-1}]_{\ge l}$, set 
$$G_{l, \Fc} := \{f_T \mid T \in \stab{l, \lambda} \, \text{for} \, \lambda \in \Fc \},$$
and let $I_{l, \Fc} \subset S$ be the ideal generated by $G_{l, \Fc}$, equivalently, 
$$I_{l,\Fc}:=\sum_{\lambda \in \Fc} I_{l, \lambda}. $$
In particular, for $\lambda \in [P_{n+l-1}]_{\ge l}$, $\Fc_\lambda:=\{\mu \in [P_{n+l-1}]_{\ge l} \mid \mu  \unlhd \lambda \}$ is a lower filter, and we have $I_{l, \lambda}=I_{l, \Fc_\lambda}$ by Lemma~\ref{inclusion}.   
For convenience, set $G_{l, \emptyset}=\emptyset$ and $I_{l, \emptyset}=(0)$. 

For an {\it upper} filter  $\emptyset \ne \Fc \subset [P_{n+l-1}]_{\ge l}$, we consider the ideal 
 $$J_{l, \Fc} :=( f \in S \mid f(\bm a)=0 \, \text{for $\forall \bm a \in K^n$ with $\Lambda_l(\bm a) \in \Fc$} ). $$
Clearly, $J_{l, \Fc}$ is a radical ideal.

\begin{theorem}\label{main1}
Let $\Fc \subsetneq [P_{n+l-1}]_{\ge l}$ be a lower filter, and $\Fc^\cmpl:=  [P_{n+l-1}]_{\ge l} \setminus \Fc$ its complement (note that $\Fc^\cmpl$ is an upper filter). Then $G_{l, \Fc}$ is a Gr\"obner basis of $J_{l, \Fc^\cmpl}$. 
\end{theorem}

The following corollary is immediate from the theorem. 

\begin{corollary}\label{Cor to main1}
With the above situation, we have $I_{l, \Fc} =J_{l, \Fc^\cmpl}$, and $I_{l, \Fc}$ is a radical ideal. 
In particular, $I_{l, \lambda}$ is a radical ideal with 
$$I_{l, \lambda} = ( f \in S \mid f(\bm a)=0 \, \text{for $\forall \bm a \in K^n$ with $\Lambda_l(\bm a) \! \! \not \! \unlhd \lambda$} ),$$
for which $\{ f_T \mid T \in  \stab{l,\mu}, \mu \unlhd \lambda  \}$ forms a Gr\"obner basis.  
\end{corollary}

The strategy of the proof of Theorem~\ref{main1} is essentially same as that of \cite[Theorem~1.1]{MOY}, but we repeat it here for the reader's convenience.  
For a partition $\lambda=(\lambda_1,\ldots,\lambda_p) \in P_m$ and a positive integer $i$,
we write $\lambda+\add i$ for the partition of $m+1$ obtained by rearranging the sequence $(\lambda_1,\dots,\lambda_i+1,\dots,\lambda_p)$, 
where we set $\lambda+\add i=(\lambda_1,\dots,\lambda_p,1)$ when $i>p$. 
For example $(4,2,2,1)+\add 2=(4,2,2,1)+\add 3=(4,3,2,1)$, and $(4,2,2,1)+\add i=(4,2,2,1,1)$ for all $i \ge 5$.
Since $\lambda  \unlhd \mu$ implies $\lambda + \add i  \unlhd \mu + \add i$ for all $i$,  if  $\Fc \subset P_m$ is an upper (resp. lower) filter, then so is 
$$
\Fc_i := \{ \mu \in P_{m-1}  \mid  \mu + \add i \in \Fc\}.
$$


\begin{example}
Even if  a lower filter $\Fc$ has a unique maximal element, $\Fc_i$ does not in general.  For example, if $\Fc :=\{ \lambda \in [P_7]_{\ge l} \mid \lambda \unlhd (3,2,2) \}$ for $l=1,2$, then $\Fc_2$ has two  maximal elements $(3,1,1,1)$ and $(2,2,2)$. 
\end{example}

\begin{lemma}[{c.f. \cite[Lemma~3.3]{MOY}}]
\label{keylemma}
Let $\emptyset \ne \Fc \subset [P_{n+l-1}]_{\ge l}$ be an upper filter, and let $f$ be a polynomial in $J_{l, \Fc}$
of the form
$$
f= g_d x_n^d  + \cdots +  g_1 x_n +g_0,
$$
where $g_0,\dots,g_d \in K[x_1,\dots,x_{n-1}]$ and $g_d\ne 0$.
Then $g_0,\dots,g_d$ belong to $J_{l, \Fc_{d+1}}$.
\end{lemma}

\begin{proof}
Let $\lambda=(\lambda_1,\dots,\lambda_p) \in \Fc_{d+1}$, and take $\bm a=(a_1, \ldots, a_{n-1}) \in K^{n-1}$ with $\Lambda_l(\bm a) =\lambda$.  Then there are  distinct elements $\alpha_1, \ldots, \alpha_p \in K$ such that $\alpha_i$ appears $\lambda_i$ times in $\bm a_{(l)}$ for $i=1, \ldots,p$. Since $\Fc$ is an upper filter, we have $\lambda + \add i \in \Fc$ for $i=1,2,\ldots,d+1$. We will consider two cases as follows (in the sequel, for $\alpha \in K$, $(\bm a, \alpha)$ means the point in $K^n$ whose coordinate is $(a_1, \ldots, a_{n-1}, \alpha)$): (i) If $p<d+1$, then $\lambda + \add{d+1} =(\lambda_1,\ldots,\lambda_p,1)$. Thus, for any $\alpha \in K \setminus \left\{\alpha_1, \alpha_2, \ldots, \alpha_p \right\}$, we have $\Lambda_l(\bm a,\alpha) =\lambda + \add{d+1}\in \Fc$, and hence $f(\bm a,\alpha)=0$.
(ii) If $p \geq d+1$, then we have $\Lambda_l(\bm  a, \alpha_i)=\lambda + \add{i} \in \Fc$ for any $i=1,\ldots,d+1$ (note that $\lambda + \add{i} \unrhd \lambda + \add{d+1} \in \Fc$ for these $i$), and hence $f(\bm a,\alpha_i)=0$.

In both cases, it follows that the polynomial $f(\bm a, x_n)=\sum^{d}_{i=0}g_i(\bm a)x^i_{n} \in K[x_n]$ has at least $d+1$ zeros. Since the degree of $f(\bm a, x_n)$ is $d$, $f(\bm a, x_n)$ is the zero polynomial in $K[x_n]$. Thus, $g_i(\bm a)=0$ for $i=0,1,\ldots,d$. Hence, $g_0,\dots,g_d \in J_{l, \Fc_{d+1}}$.
\end{proof}

\noindent{\it The proof of Theorem~\ref{main1}.} First, we show that $G_{l, \Fc} \subset J_{l, \Fc^\cmpl}$. Take $T \in \stab{l, \lambda}$ for  $\lambda \in \Fc$, and $\bm a \in K^n$ with $\Lambda_l(\bm a) \in \Fc^\cmpl$ (i.e., $\Lambda_l(\bm a) \not \in \Fc$). Since $\Fc$ is a lower filter, we have $\Lambda_l(\bm a) \not \!\! \unlhd \, \lambda$, and hence $f_T(\bm a)=0$ by Lemma~\ref{f_T(a)=0}.  
So $f_T \in  J_{l, \Fc^\cmpl}$. 

For $\mu \in [P_{n+l-2}]_{\geq l}$, it is easy to see that   
$$\mu \not \in  (\Fc^\cmpl)_{i} \Longleftrightarrow \mu+\add i \not \in \Fc^\cmpl
 \Longleftrightarrow \mu+\add i \in \Fc  \Longleftrightarrow \mu \in \Fc_i,$$
so we have $[P_{n+l-2}]_{\geq l} \setminus  (\Fc^\cmpl)_{i} =\Fc_i$. 

To prove the theorem, it suffices to show that the initial monomial $\ini{f}$ for all $0 \ne f \in J_{l, \Fc^\cmpl}$ can be divided by $\ini{f_T}$ for some $f_T \in G_{l, \Fc}$. We will prove this by induction on $n$. The case $n=1$ is trivial. For $n\geq2$, let $f= g_d x_n^d  + \cdots +  g_1 x_n +g_0 \in J_{l, \Fc^\cmpl}$, where $g_i \in K[x_1,\ldots,x_{n-1}]$ and $g_d\neq0$. By Lemma \ref{keylemma}, one has $g_d \in J_{l, (\Fc^\cmpl)_{d+1}}$. By the induction hypothesis, we have $G_{l, \Fc_{d+1}} (=G_{l,[P_{n+l-2}]_{\geq l} \setminus (\Fc^\cmpl)_{d+1}})$ is a Gr\"obner basis of $J_{l,(\Fc^\cmpl)_{d+1}}$. Then there is $T \in \stab{l, \mu}$ for  $\mu \in \Fc_{d+1}$ such that $\ini{f_T}$ divides $\ini{g_d}$. 
Set $\lambda :=\mu + \add{d+1} \in \Fc$. Let us consider the tableau $T' \in \stab{l, \lambda}$ such that the image of each $i = 1,2,\dots,n-1$ is same for $T$ and $T'$. So $n$ is in the  newly added square. Since $\lambda=\mu + \add{d+1}$, $n$ is in the $(q+1)$-st row of $T'$ for some $q \le d$. 
Since we have $\ini{f}=\ini{g_d x_n^d}=x_n^d \cdot \ini{g_d}$ and $\ini{f_{T'}}=x_n^q \cdot \ini{f_T}$ by \eqref{in(f)}, $\ini{f_{T'}}$ divides $\ini{f}$. Hence, the proof is completed. 
\qed 

\medskip

For  $\lambda=(\lambda_1, \ldots, \lambda_p) \in [P_{n+l-1}]_{\ge l}$, set $H_{l, \lambda}:=\{ \bm a \in K^n \mid \Lambda_l(\bm a) =\lambda\}$. Then we have the decomposition $K^n=\bigsqcup_{\lambda \in [P_{n+l-1}]_{\ge l}} H_{l, \lambda}$, and the dimension of $H_{l, \lambda}$ equals the length $p$ of $\lambda$. For an upper filter $\Fc \subset [P_{n+l-1}]_{\ge l}$,  $S/J_{l, \Fc} \, (=S/I_{l, \Fc^\cmpl} )$ is the coordinate ring of $\bigsqcup_{\lambda \in \Fc} H_{l, \lambda}$. 

\begin{proposition}
The codimension of the ideal $I_{l, \lambda}$ is $\lambda_1-l+1$. 
\end{proposition}

\begin{proof}
By the above remark, the algebraic set defined by $I_{l, \lambda}$ is the union of $H_{l, \mu}$ for all $\mu \in [P_{n+l-1}]_{\ge l}$ with $\mu \not \!\! \unlhd \lambda$. Among these partitions, $\mu'=(\lambda_1+1, 1, 1, \ldots)$ has the largest length $n+l-1-\lambda_1$, and hence $\operatorname{codim} I_{l, \lambda}= n- \dim S/I_{l, \lambda}=n-(n+l-1-\lambda_1)=\lambda_1-l+1$. 
\end{proof}

\begin{example}
For $\lambda=(3,3,1)$, $\stab{2, \lambda}$ consists of the following 11 elements 
$$
\ytableausetup{mathmode, boxsize=1.2em}
\begin{ytableau}
1 & 1 & 2  \\
3 & 4& 5    \\
6  \\  
\end{ytableau}, 
\quad  
\begin{ytableau}
1 & 1 & 2  \\
3 & 4& 6    \\
5 \\  
\end{ytableau},
\quad  
\begin{ytableau}
1 & 1 & 2  \\
3 & 5& 6    \\
4 \\  
\end{ytableau},
\quad  
\begin{ytableau}
1 & 1 & 3  \\
2 & 4& 5     \\
6\\  
\end{ytableau}, 
\quad  
\begin{ytableau}
1 & 1 & 3  \\
2 &  4 & 6     \\
5\\  
\end{ytableau}, 
\quad  
\begin{ytableau}
1 & 1 & 3  \\
2 &  5 & 6     \\
4\\  
\end{ytableau}, 
$$
$$
\ytableausetup{mathmode, boxsize=1.2em}
\begin{ytableau}
1 & 1 & 4  \\
2 & 3& 5    \\
6  \\  
\end{ytableau}, \quad 
\begin{ytableau}
1 & 1 & 4  \\
2 & 3& 6   \\
5  \\  
\end{ytableau}, \quad 
\begin{ytableau}
1 & 1 & 4  \\
2 & 5& 6    \\
3  \\  
\end{ytableau},  \quad 
\begin{ytableau}
1 & 1 & 5  \\
2 & 3& 6   \\
4  \\  
\end{ytableau},  \quad 
\begin{ytableau}
1 & 1 & 5  \\
2 & 4& 6   \\
3  \\  
\end{ytableau}, 
$$
so $I_{2, \lambda}$ is minimally generated by 11 elements. 
For a non-empty subset $F \subset [n]$, consider the ideal $P_F=(x_i-x_j \mid i, j \in F)$. Clearly, $P_F$ is a prime ideal of codimension $\# F-1$. 
By Corollary~\ref{Cor to main1}, $I_{3, \lambda}$ is a radical ideal whose minimal primes are  $P_F$ for $F \subset [n]$ either (i) $1 \in F$ and $\# F =3$, or  (ii) $1 \not \in F$ and $\# F =4$. 
\end{example}

\section{Under the opposite monomial order}
Philosophically, we next treat  the Gr\" obner basis of $I_{l,\lambda}$ with respect to the lexicographic order with $x_1 >x_2 >\cdots >x_n$, which is opposite to the one used in the previous section.  However, for notational simplicity, we keep using the lexicographic order with $x_1 < \cdots < x_n$, but we consider tableaux whose squares are bijectively filled by the multiset $\{1, \ldots, n-1, {\overbrace{n, \ldots, n}^{\text{$l$ copies}}}\}$.  For $\lambda \in [P_{n+l-1}]_{\ge l}$,  let  $\tab{\lambda, l}$ be the set of such tableaux of shape $\lambda$. As in the previous section, we can define the standard-ness of $T \in \tab{\lambda, l}$.   For example, the tableau $T$ in \eqref{stab(lambda, l)} below is standard.  Let $\stab{\lambda, l}$ be the subset of $\tab{\lambda, l}$ consisting of standard tableaux.  
By the same argument as Lemma~\ref{stab}, we have the following.

\begin{lemma}\label{stab2} 
For $\lambda \in [P_{n+l-1}]_{\ge l}$, $\{ f_T \mid T\in \stab{\lambda, l} \}$ forms a basis of the vector space spanned by  $\{ f_T \mid T\in \tab{\lambda,l} \}$. 
\end{lemma}

For $\lambda \in [P_{n+l-1}]_{\ge l}$, consider the ideal  
$$I_{\lambda,l}:=(f_T \mid T \in \tab{\lambda,l}) =(f_T \mid T \in \stab{\lambda,l} ).$$

For $\bm a \in K^n$, set 
$\bm a^{(l)}:=(a_1, \ldots, a_{n-1}, \overbrace{a_n, \ldots, a_n}^{\text{$l$ copies}}) \in K^{n+l-1}$ and $\Lambda^l(\bm a):=\Lambda(\bm a^{(l)}) \in [P_{n+l-1}]_{\ge l}$. 
Up to the automorphism of $S$ exchanging $x_1$ and $x_n$, the ideal $I_{\lambda, l}$ coincides with $I_{l,\lambda}$ treated in the previous section. 
Hence we have $I_{\mu, l} \subset I_{\lambda,l}$ for $\mu, \lambda \in [P_{n+l-1}]_{\ge l}$ with $\mu \unlhd \lambda$, and 
\begin{equation}\label{I_{F,l}=J}
I_{\lambda, l} =( f \in S \mid f(\bm a)=0 \, \text{for $\forall \bm a \in K^n$ with $\Lambda^l(\bm a) \! \! \not \! \unlhd \lambda$}).
\end{equation}

In $T \in \stab{\lambda,l}$, all $n$'s are in the bottom of their columns.  Let $w(T)$ denote the number of squares which locate above some $n$.  For example, if 
\begin{equation}\label{stab(lambda, l)}
T=\ytableausetup{mathmode, boxsize=1.2em}
\begin{ytableau}
1& 2& 3 & 5& 8 & 8  \\
4 & 6   & 8  \\
7 & 8\\
\end{ytableau}
\end{equation}
($n=8$ and $l=4$ in this case), then $w(T)=3$. In fact, the squares filled by 2, 6, and 3 are counted. 
For $T \in \tab{\lambda, l}$, the degree of $\ini{f_T}$ with respect to $x_n$ is $w(T)$. 
Let $\sh_{<n}(T) \in P_{n-1}$ denote the shape of $T'$, where  $T'$ is the tableau obtained by removing all squares filled by $n$ from $T$.  For example, if $T$ is the above one, we have $\sh_{<8}(T)=(4,2,1)$. 

For $\mu \in P_{n-1}$, set 
$$\< \mu \>^l := \{ \lambda \in [P_{n+l-1}]_{\ge l} \mid \text{$\exists T \in \stab{\lambda,l}$ with $\sh_{<n}(T)=\mu$}\}. $$
For $\lambda \in \< \mu \>^l$ and $T \in \stab{\lambda,l}$ with $\sh_{<n}(T)=\mu$, the positions of the squares filled by $n$ only depend on $\lambda$ and $\mu$. 
We call these squares {\it $n$-squares} of $\lambda$. 
Similarly, $w(T)$ does  not depend on a particular choice of $T$, and we denote this value by $w_\mu(\lambda)$. 

\begin{lemma}\label{maximum element}
Let $\lambda \in [P_{n+l-1}]_{\ge l}$, and $\Fc:=\{ \rho \in [P_{n+l-1}]_{\ge l} \mid \rho \unlhd \lambda \}$ the lower filter of $[P_{n+l-1}]_{\ge l}$. 
For $\mu \in P_{n-1}$, if $X:= \<\mu\>^l \cap \Fc$ is non-empty, then there is the element $\wlam \in X$ satisfying $w_\mu(\rho) > w_\mu(\wlam)$ for all $\rho \in X \setminus \{ \wlam \}$.    
\end{lemma}

\begin{proof}
We determine $\wlam =(\wlam_1, \wlam_2, \ldots, \wlam_p)$ inductively from $\wlam_1$.  
First, we set 
$$\wlam_1:=\max \{ \, \rho_1 \mid \rho=(\rho_1, \ldots, \rho_q)  \in X \, \} \quad \text{and} \quad  X_1:= \{\,  \rho \in X \mid \rho_1 =\wlam_1 \, \},$$ 
and next 
$\wlam_2:=\max \{ \rho_2 \mid \rho \in X_1 \}$ and $X_2 := \{ \rho \in X_1 \mid \rho_2 =\wlam_2 \}.$ 
We repeat this procedure until the sum $\wlam_1+\wlam_2 +\cdots$ reaches $n+l-1$. 

We will show that $\wlam$ has the expected property. 
For $\rho \in X \setminus \{ \wlam \}$, set $i_0:=\min\{ i \mid \rho_i \ne \lambda_i\}$. 
Then we have $\rho_{i_0} < \wlam_{i_0}$, and $\rho$ has an $n$-square in the $i$-th row for some $i >i_0$. 
Let $i_1$ be the smallest  $i$ with this property. Raising up the right most  square in the $i_1$-th row to the right end of $i_0$-th row, we get $\rho' \in X$ 
(here we use the present form of $\Fc$). 
It is clear that $w_\mu(\rho) > w_\mu(\rho')$ and $\rho \lhd \rho'$. Repeating this argument until our partition  reaches $\wlam$, we get the expected inequality.  
\end{proof}

\begin{example}
In the above lemma, the case $\lambda \ne \wlam$ might happen. For example, if $\lambda=(4,2,1)$ and $\mu=(3,3)$ (so $l=1$ now), we have $\wlam=(3,3,1)$. 
\end{example}

For a lower filter $\Fc \subset [P_{n+l-1}]_{\ge l}$ and a non-negative integer $k$, set 
$$\Fc^k := \{ \, \mu \in P_{n-1} \mid \text{$\exists \lambda \in \<\mu\>^l \cap \Fc$ with $w_\mu(\lambda) \le k$}  \, \}.$$

\begin{example}
Consider the case $n=6, l=2$, $\lambda=(3,3,1)$ and $\Fc:=\{ \, \rho \in [P_7]_{\ge 2} \mid 
\rho \unlhd \lambda \, \}$. Then $\Fc^3, \Fc^2,\Fc^1, \Fc^0$ are the lower filters of $P_5$ whose unique maximal elements are $(3,2), (3,1,1), (2,1,1,1)$ and $(1,1,1,1,1)$, respectively. In the following diagrams, $\star$'s represent the positions of $n$-squares of the corresponding partitions of $n+l-1 \, (=7)$. It is also easy to see that $\Fc^k=\Fc^3$ for all $k \ge 3$. 
$$
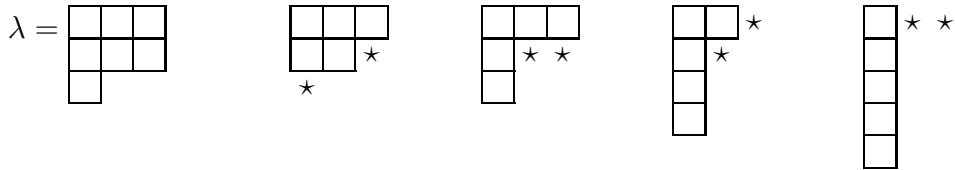
\lambda=
\ytableausetup{mathmode, boxsize=1em}
\begin{ytableau}
 {} & {} & {}   \\
 {} & {} &  {}  \\
{} \\
\end{ytableau}
\qquad \qquad
\ytableausetup{mathmode, boxsize=1em}
\begin{ytableau}
 {} & {} & {}   \\
 {} & {} &  \none[\star]  \\
\none[\star] \\
\end{ytableau}
\qquad \quad 
\ytableausetup{mathmode, boxsize=1em}
\begin{ytableau}
 {} & {} &   {}  \\
 {} &  \none[\star]  &  \none[\star]   \\
{} \\
\end{ytableau}
\qquad \quad 
\ytableausetup{mathmode, boxsize=1em}
\begin{ytableau}
 {} & {} &   \none[\star]  \\
 {} & \none[\star]    \\
{}  \\
{}  \\
\end{ytableau}
\qquad \quad 
\ytableausetup{mathmode, boxsize=1em}
\begin{ytableau}
 {} & \none[\star]   &   \none[\star]  \\
 {} \\
{}  \\
{}  \\
{} \\
\end{ytableau}
$$
\end{example}

\begin{lemma}\label{filter filter}
If $\Fc \subset [P_{n+l-1}]_{\ge l}$ is a lower filter,  then $\Fc^k$ is a lower filter of $P_{n-1}$. 
\end{lemma}

\begin{proof}
It suffices to show that if $\mu \in \Fc^k$ covers $\nu \in P_{n-1}$ then $\nu \in  \Fc^k$.  In this situation, there are two integers $j, j'$ with $j< j'$ such that $\nu^\perp_j=\mu^\perp_j+1$, $\nu^\perp_{j'}=\mu^\perp_{j'}-1$, and $\nu^\perp_i=\mu^\perp_i$ for all $i \ne j, j'$. In other words,  moving a square in the $j'$-th column of $\mu$ to the $j$-th column, we get $\nu$. Anyway, we can take $\lambda \in \Fc \cap \<\mu \>^l$ with $w_\mu(\lambda) \le k$,  and we want to construct $\rho \in \Fc \cap \<\nu \>^l$ with $w_\nu(\rho) \le k$. 

For each $i$, we have $\mu_i^\perp \le \lambda_i^\perp \le \mu_i^\perp+1$, and there is an $n$-square in the $i$-th column of $\lambda$ if and only if  $\lambda_i^\perp = \mu_i^\perp+1$. We have the following four cases.
\begin{itemize}
\item[(1)] $\lambda^\perp_j=\mu^\perp_j$ and  $\lambda^\perp_{j'}=\mu^\perp_{j'}$. 
\item[(2)] $\lambda^\perp_j=\mu^\perp_j+1$ and  $\lambda^\perp_{j'}=\mu^\perp_{j'}+1$. 
\item[(3)] $\lambda^\perp_j=\mu^\perp_j$ and  $\lambda^\perp_{j'}=\mu^\perp_{j'}+1$.
\item[(4)] $\lambda^\perp_j=\mu^\perp_j+1$ and  $\lambda^\perp_{j'}=\mu^\perp_{j'}$.  
\end{itemize} 
In the case (4), we set $\rho=\lambda$, that is,  exchanging the $n$-square in the $j$-th column and the bottom square of $j'$-th column, we get $\rho$ and $\nu$ from $\lambda$ and $\mu$.  In the other cases, we first move the $n$-squares in the $j$-th and $j'$-th columns of $\lambda$ (their existence depends on the cases (1)-(3))  vertically along the change from $\mu$ to $\nu$. 
For example, in the case (2),  the above operation is  
\begin{equation}\label{move of n squares}
 \ytableausetup{mathmode, boxsize=1.1em}
\begin{ytableau}
\none[\text{\tiny{$j$}}] & \none & \none[\text{\tiny{$j'$}}] \\
 {} & \none[\hdots] & {}   \\
 {} & \none[\hdots]  & {} \\
{}  &\none & n\\
n
\end{ytableau}
\quad 
\begin{ytableau}
\none \\
 \none   \\
\none[\too] \\
\none \\ 
\none
\end{ytableau}
\quad 
 \ytableausetup{mathmode, boxsize=1.1em}
\begin{ytableau}
\none[\text{\tiny{$j$}}] & \none & \none[\text{\tiny{$j'$}}] \\
 {} & \none[\hdots] & {}   \\
 {} & \none[\hdots]  & n\\
  \\
 \\
n
\end{ytableau}
\end{equation}
 Furthermore, if necessary, 
we apply a suitable column permutation as the following figure (in this situation, since $\mu^\perp_{j-1}=\nu^\perp_{j-1} \ge \nu^\perp_j=\mu^\perp_j +1$, we have $\mu^\perp_{j-1} > \mu^\perp_j$ and there is no $n$-square in the $(j-1)$-st column of the left and the middle diagrams).   
In any cases, we have $\rho \unlhd \lambda \in \Fc$ and $w_\nu(\rho) \le w_\mu(\lambda) \le k$, that is,  $\rho$ satisfies the expected property. 
$$
\ytableausetup{mathmode, boxsize=1.1em}
\begin{ytableau}
\none[\vdots] & \none[\vdots] & \none[\tiny{\text{$j$}}] \\
 {} & {} & {}   \\
 {} & {} & n \\
{} \\
\none[\vdots] \\
\end{ytableau}
\qquad 
\begin{ytableau}
\none \\
 \none   \\
\none[\too] \\
\none \\ 
\none
\end{ytableau}
\qquad 
\begin{ytableau}
\none[\vdots] & \none[\vdots] & \none[\tiny{\text{$j$}}] \\
 {} & {} & {}   \\
 {} & {} & {} \\
{} & \none & n \\
\none[\vdots] \\
\end{ytableau}
\qquad 
\begin{ytableau}
\none \\
 \none   \\
\none[\too] \\
\none \\ 
\none
\end{ytableau}
\qquad 
\begin{ytableau}
\none[\vdots] & \none[\vdots] & \none[\tiny{\text{$j$}}] \\
 {} & {} & {}   \\
 {} & {} & {} \\
{} & n \\
\none[\vdots] \\
\end{ytableau}
$$
\end{proof}

\begin{proposition}\label{keylemma2} 
Let $\lambda \in [P_{n+l-1}]_{\ge l}$, and $\Fc:=\{ \rho \in [P_{n+l-1}]_{\ge l} \mid \rho \unlhd \lambda \}$ the lower filter of $[P_{n+l-1}]_{\ge l}$. 
If $f \in I_{\lambda,l}$ is of the form $f= g_d x_n^d  + \cdots +  g_1 x_n +g_0 $ with $g_0,\dots,g_d \in K[x_1,\dots,x_{n-1}]$ and $g_d \ne 0$,
then $g_0,\dots,g_d$ belong to $I_{\Fc^d}$. 
\end{proposition}

\begin{proof}
Assume that $g_m \not \in I_{\Fc^d}$ for some $m$. By the classical case (i.e., when $l=1$) of Corollary~\ref{Cor to main1}, there are some $\bm a \in K^{n-1}$ such that $\mu :=\Lambda(\bm a) \not \in \Fc^d$ and $g_m(\bm a) \ne 0$.  If $\mu =(\mu_1, \ldots, \mu_p)$, there are distinct elements $\alpha_1, \ldots, \alpha_p \in K$ such that $\alpha_i$ appears $\mu_i$ times in $\bm a$ for $i=1, \ldots, p$. 

We have 
\begin{equation}\label{expression}
f= \sum_{\substack{T \in \stab{\lambda', l} \\ \lambda' \in \Fc}}h_T \cdot f_T
\end{equation}
for some $h_T \in S$. For $T\in \stab{\lambda', l}$, replacing $i$ with $a_i$ in $T$ for all $1 \le i \le n-1$, and $n$ with $x_n$, we get the tableau $T(\bm a)$ whose entries are elements of $K \cup \{ x_n \}$. 

Take $\rho \in \< \nu \>^l$ for some $\nu \in P_{n-1}$. We call  a bijective filling $\Tc $of the squares of the Young diagram of  $\rho$ by the multiset
$$\{  \, \overbrace{\alpha_1, \ldots, \alpha_1}^{\text{$\mu_1$ copies}}, \overbrace{\alpha_2, \ldots, \alpha_2}^{\text{$\mu_2$ copies}},\cdots, \overbrace{\alpha_p, \ldots, \alpha_p}^{\text{$\mu_p$ copies}},  \overbrace{x_n, \ldots, x_n}^{\text{$l$ copies}} \, \}$$ 
such that all $n$-squares  are filled by $x_n$ is called an {\it $\bm a$-tableau}. We call $\rho$ the {\it shape} of $\Tc$, and denote it by $\sh(\Tc)$. We also denote $\nu$ by $\sh_{<n}(\Tc)$. 
A typical example of an $\bm a$-tableau is $T(\bm a)$ given above. We say an $\bm a$-tableau $\Tc$ is {\it regular}, if the entries in the $j$-th column of $\Tc$ are all distinct for each $j$. 
Note that $T(\bm a)$ is regular if and only if $f_T(\bm a, x_n) \ne 0$. 

For all $\alpha \in K$, we can show that $\bar{\mu}:=\Lambda^l(\bm a, \alpha)$ belongs to $\< \mu \>^l$ 
(recall that $\mu=\Lambda(\bm a)$). For example, if $\alpha \not \in \{ \alpha_1, \ldots, \alpha_p \}$, we have $\bar{\mu}=(\mu_1, \ldots, \mu_j, l, \mu_{j+1}, \ldots, \mu_p)$, where $j:=\max \{ i \mid \mu_i \ge l \}$, and hence $\bar{\mu}_i \ge \mu_i$ and $\mu_i^\perp \le \bar{\mu}_i^\perp \le \mu_i^\perp+1$ for all $i$.  The case  $\alpha \in \{ \alpha_1, \ldots, \alpha_p \}$ can be shown by a similar argument. 
Anyway, if $\< \mu \>^l \cap \Fc = \emptyset$,  then $\Lambda^l(\bm a, \alpha) \not \in \Fc$, and hence $f(\bm a, \alpha)=0$  by \eqref{I_{F,l}=J}.  So it implies that $f(\bm a, x_n)=0$ and $g_m(\bm a)=0$.  This is a contradiction. So $\< \mu \>^l \cap \Fc$ is non-empty, and it has the element $\wlam$ with the minimum $w_{\mu}(-)$ by Lemma~\ref{maximum element}.  Let $\wTc$ be an $\bm a$-tableau of shape $\wlam$ with $\sh_{<n}(\wTc)=\mu$ such that all squares in the $i$-th row of $\mu$ are filled by $\alpha_i$. 
Assume that, for each $i$ with $1 \le i \le p$, $\alpha_i$ appears $d_i$ times in squares above some $n$-squares in $\wTc$.  See Example~\ref{a-tableau} below. 
 We have $\sum_{i=1}^p d_i =w_\mu(\lambda) >d$, where the inequality follows from that $\mu \not \in \Fc^d$. 

\medskip

\noindent{\bf Claim.} Let $\Tc$ be a regular $\bm a$-tableau with $\sh(\Tc) \in \Fc$.  For all $i$ with $1 \le i \le p$, $\alpha_i$ appears at least $d_i$ times in squares above some $n$-squares in $\Tc$.  

\medskip

\noindent{\it Proof of Claim.} Set $\nu := \sh_{<n}(\Tc) \in P_{n-1}$. 
We will prove the assertion by induction on $\nu$ with respect to the dominance order. 
Since $\Tc$ is regular, it is easy to see that  $\mu \unlhd \nu$ by the classical case (i.e., when $l=1$) of Corollary~\ref{Cor to main1}. 
If  $\mu=\nu$, applying suitable actions of column stabilizers (i.e., permutations of entries in the same column), we may assume that each square in the $i$-th row of $\Tc$ is filled by $\alpha_i$ or $x_n$. So the assertion can be shown by an argument similar to the proof of Lemma~\ref{maximum element}.  Next consider the case $\mu \lhd \nu$.  As the induction hypothesis, we assume that the assertion holds for $\Tc'$ with $\mu \unlhd \sh_{<n}(\Tc') \lhd \nu$. 

To proceed with proof by contradiction, assume that $\Tc$ does not satisfy the expected condition, that is, there is some $s$ such that $\alpha_s$ appears less than $d_s$ times in squares above some $n$-squares in $\Tc$.  Since $\mu \lhd \nu$ now, there are some $t$ and $j,j'$ with $j <j'$ such that $\alpha_t$ appears in the $j'$-th column of  $\Tc$,  but does not appear in the $j$-th column. If $\alpha_s$ has this property, we take $s$ as $t$. 
We move the square in the $j'$-th column filled by $\alpha_t$ to the $j$-th column, 
and get the partition $\nu' \in P_{n-1}$ (a suitable column permutation might be required). 
The following condition is crucial. 
\begin{itemize}
\item[$(*)$] $s=t$ holds, and the bottom of the $j$-th column of $\Tc$ is an $n$-square, and that of the $j'$-th column is not.
\end{itemize}
In the case $(*)$ is not satisfied, we move the $n$-squares in these columns (if they exist) vertically like \eqref{move of n squares}, then apply a suitable column permutation if necessary (sometimes, we have to move the $j$-th column to left and/or the $j'$-th column to right).  Finally, we get an $\bm a$-tableau $\Tc'$ with $\sh_{<n}(\Tc')=\nu'$.  
On the other hand, if $(*)$ holds, we move the $n$-square in the $j$-th column to below the bottom of the $j'$-th column. Of course, do not forget to move the $\alpha_s$-square in the $j'$-th column to the $j$-th column.  Applying a suitable column permutation if necessary, we get $\Tc'$ with $\sh_{<n}(\Tc')=\nu'$.  In this case, we have $\sh(\Tc)=\sh(\Tc')$. 

In both cases, $\Tc'$ is regular, and $\alpha_s$ appears less than $d_s$ times in squares above some $n$-squares in $\Tc'$. Moreover, we have $\sh(\Tc') \unlhd \sh(\Tc)$, and hence $\sh(\Tc') \in \Fc$. 
Since $(\mu \unlhd) \, \nu' \lhd \nu$, it contradicts the induction hypothesis. \qed 

\medskip

We back to the proof of the proposition itself.  In \eqref{expression}, we have 
$$
f(\bm a, x_n)= \sum
h_T(\bm a, x_n) \cdot f_T(\bm a, x_n)
$$
If $f_T(\bm a, x_n) \ne 0$, then $T(\bm a)$ is regular. So, by Claim, $f_T(\bm a, x_n)$ can be divided by 
$$R(x_n)=\prod_{1 \le i \le p}(x_n-\alpha_i)^{d_i},$$ and $f(\bm a, x_n)$ itself can be divided by 
$R(x_n)$.  While the degree of $f(\bm a, x_n)$ is at most $d$, we have $\deg f(\bm a, x_n) \ge \deg R(x_n) = \sum_{i=1}^p d_i >d$.  This is a contradiction. 
\end{proof}

\begin{example}\label{a-tableau}
Consider the case $n=8, l=3$, and take the lower filter given by 
$\Fc =\{ \lambda \in [P_{10}]_{\ge 3} \mid \lambda \unlhd (4,4,2)\}$.  
If $\bm a =(\alpha_1, \alpha_1, \alpha_1, \alpha_2, \alpha_2, \alpha_3, \alpha_3)$ (hence $\mu=(3,2,2)$), the $\bm a$-tableau $\wTc$ given in the proof of Proposition~\ref{keylemma2} is as follows 
$$
\ytableausetup{mathmode, boxsize=1.3em}
\begin{ytableau}
\alpha_1 & \alpha_1 & \alpha_1 & x_8  \\
\alpha_2 & \alpha_2   & x_8  \\
\alpha_3 & \alpha_3 \\
x_8
\end{ytableau}
$$
Above three $n\, (=8)$-boxes, there are two copies of $\alpha_1$, so we have $d_1=2$. 
Similarly, since there is one $\alpha_2$ (resp. $\alpha_3$) above  three $n$-boxes, we have $d_2=1$ (resp. $d_3=1$). 

The following are examples of regular $\bm a$-tableaux whose shape belong to $\Fc$.
In each case, there are at least 2 (resp. 1) $\alpha_1$ (resp. $\alpha_2$ and $\alpha_3$)  above $n$-squares. 
$$
\ytableausetup{mathmode, boxsize=1.3em}
\begin{ytableau}
\alpha_1 & \alpha_1 & \alpha_1 & x_8  \\
\alpha_2 & \alpha_2    \\
\alpha_3 & \alpha_3 \\
x_8 & x_8 
\end{ytableau}
\qquad \qquad 
\begin{ytableau}
\alpha_1 & \alpha_1 & \alpha_1 & \alpha_2 \\
\alpha_2 & \alpha_3 & x_8 & x_8  \\
\alpha_3 &  x_8
\end{ytableau}
\qquad \qquad 
\begin{ytableau}
\alpha_1 & \alpha_1 & \alpha_1 & \alpha_2 \\
\alpha_2 & \alpha_3 & \alpha_3 & x_8  \\
x_8 &  x_8
\end{ytableau}
$$
\end{example}

\begin{theorem}\label{main1.5}
For $\lambda \in [P_{n+l-1}]_{\ge l}$, set $\Fc:=\{ \rho \in [P_{n+l-1}]_{\ge l} \mid \rho \unlhd \lambda \}$ be the lower filter. Then $\{ f_T \mid T \in  \stab{\rho,l}, \rho \in \Fc  \}$ is a Gr\"obner basis of $I_{\lambda,l}$.  
\end{theorem}

\begin{proof}
It suffices to show that the initial monomial $\ini{f}$ for all $0 \ne f \in I_{\lambda,l}$ can be divided by $\ini{f_T}$ for some $T \in  \stab{\rho,l}$ with $\rho \in \Fc$. Let $f= g_d x_n^d  + \cdots +  g_1 x_n +g_0$, where $g_i \in K[x_1,\ldots,x_{n-1}]$ and $g_d\neq0$. By Proposition~\ref{keylemma2}, one has $g_d \in I_{\Fc^d}$.   By Theorem~\ref{HWMOY},    $\{ f_T \mid T \in  \stab{\mu}, \mu \in \Fc^d  \}$ is a Gr\"obner basis of $I_{\Fc^d}$ (since we fix the monomial order, it is enough to consider standard tableaux, see \cite[Remark~3.5]{MOY}), and there is a tableau $T \in \stab{\mu}$ for  some $\mu \in \Fc^d$ such that $\ini{f_T}$ divides $\ini{g_d}$. So we can take $\rho \in \<\mu \>^l \cap \Fc$ with $e:=w_\mu(\rho) \le d$. Let us consider the tableau $T' \in \stab{\rho,l}$ such that the image of each $i = 1, \dots,n-1$ is same for $T$ and $T'$. 
Since we have $\ini{f}=x_n^d \cdot \ini{g_d}$ and $\ini{f_{T'}}=x_n^e \cdot \ini{f_T}$, $\ini{f_{T'}}$ divides $\ini{f}$. 
\end{proof}

\begin{example}
Contrary to Theorem~\ref{main1}, Theorem~\ref{main1.5} cannot be generalized to the ideal $I_{\Fc,l}:=(f_T \mid T \in \tab{\lambda,l}, \lambda \in \Fc)$ for a lower filter $\Fc \subset [P_{n+l-1}]_{\ge l}$.  
For example, if  $\Fc \subset [P_8]_{\ge 2}$ is the lower filter whose maximal elements are $(4,2,1,1)$ and $(3,3,2)$, then $x_4^2x_5^3x_6x_7^2$ is a minimal generator of $\ini{I_{\Fc, 2}}$, but this cannot be represented in the form of $\ini{f_T}$ for $T \in \stab{\lambda,2}$. 
\end{example}

The following fact might be well-known to the specialist, and is stated in \cite{MOY} without proof. This time, we give a proof for the reader's convenience. 
\begin{lemma}\label{product of linear forms}
Let $I \subset S=K[x_1, \ldots, x_n]$ be a graded ideal, and $G \subset I$ a Gr\"obner basis of $I$ with respect to the lexicographic order $<$ with   $x_1 < x_2 <\cdots < x_n$. If all elements of $G$ are products of linear forms, then $G$ is a Gr\"obner basis of $I$ with respect to any monomial order $\prec$ with   $x_1 \prec x_2  \prec \cdots \prec x_n$. 
\end{lemma}

The assumption that $I$ is graded is unnecessary, but we add it here for the simplicity.

\begin{proof}
Since $g \in G$ is a product of linear forms, we have $\Ini_{\prec}(g)=\Ini_<(g)$,  and hence 
$$\Ini_<(I) =(\Ini_< (g) \mid g \in G) =(\Ini_\prec (g) \mid g \in G) \subset\Ini_{\prec}(I).$$
Since $\Ini_\prec(I)$ and $\Ini_<(I)$ have the same Hilbert function (in fact, they have the same Hilbert function as $I$ itself), we have $\Ini_\prec(I) =\Ini_<(I)$. It implies that $G$ is a Gr\"obner basis of $I$ with respect to $\prec$. 
\end{proof}

\begin{corollary}\label{smallest & largest}
Let $\lambda \in [P_{n+l-1}]_{\ge l}$. 
With respect to a monomial order in which $x_1$ is either the smallest or the largest among the variables $x_1, \ldots, x_n$, $\{ f_T \mid T \in  \tab{l, \rho}, \rho \unlhd \lambda   \}$ is a Gr\"obner basis of $I_{l, \lambda}$.   
\end{corollary}

Since we consider several monomial orders, we have to treat $\tab{l, \lambda}$, not $\stab{l, \lambda}$. 

\begin{proof}
First, we consider the case $x_1$ is the smallest among $x_1, \ldots, x_n$. Since the ideal $I_{l, \lambda}$ is symmetric for variables $x_2, \ldots,x_n$, and Specht polynomials are products of linear forms, we may assume that our monomial order is the lexicographic order with $x_1 < \cdots < x_n$ by Lemma~\ref{product of linear forms}, and the assertion follows from 
Corollary~\ref{Cor to main1}. Similarly, if $x_1$ is the largest, we may assume that our monomial order is the lexicographic order with $x_1 > \cdots > x_n$, and the assertion follows from Theorem~\ref{main1.5}.  
\end{proof}

\begin{example}
For $\lambda=(3,3)$, $I_{2,\lambda}$ is generated by 3 elements of degree 3. With respect to a monomial order in which $x_1$ is the smallest, $\ini{I_{2,\lambda}}$ is minimally generated by 3 elements of degree 3 and 2 elements of degree 4. On the other hand, with respect to an order in which $x_1$ is the largest,  $\ini{I_{2,\lambda}}$ is minimally generated by 3 elements of degree 3, 3 elements of degree 4, and an element of degree 6. Computer experiment suggests that $\ini{I_{l, \lambda}}$ with respect to an order  in which $x_1$ is the smallest requires fewer generators. 
\end{example}

\begin{problem}
With the notation of Corollary~\ref{smallest & largest}, is $\{ f_T \mid T \in  \tab{l, \rho}, \rho \unlhd \lambda  \}$ a universal   Gr\" obner basis of $I_{l, \lambda}$?
\end{problem}

We have computed several partitions $\lambda$ up to $n=8$ using SageMath and {\it Macaulay2}, and we have not found a counter example yet.     

\section{A generalization of the case $\#Y_1 \ge 2$ and $\#Y_2=\cdots =\#Y_l=1$}
In this section, we fix a positive integer $m$ with $1 \le m \le n$, and set 
$$\Delta_m:=\Delta(\{1, \ldots, m\})=\prod_{1 \le i < j \le m}(x_i-x_j).$$
For $T \in \tab{l, \lambda}$ with $[P_{n+l-1}]_{\ge l}$, set 
$$f_{m,T}:= \operatorname{lcm}\{ f_T, \Delta_m \} \in S \quad \text{and} \quad I_{l,m,\lambda}:=( f_{m,T} \mid T \in \tab{l, \lambda}) \subset S.$$
Note that $I_{l,1,\lambda}=I_{l,\lambda}$ and $I_{l,n,\lambda}=(\Delta_n)$.

\begin{example}
Even if $l=1$, $I_{l,m,\lambda}$ is not a radical ideal in general, while their generators are squarefree products of linear forms $(x_i-x_j)$.  For example, if $\lambda=(2,2)$, we have 
$$I_{1,3,\lambda}=(  \Delta_3\cdot(x_1-x_4), \Delta_3\cdot(x_2-x_4), \Delta_3\cdot (x_3-x_4)),$$
where $\Delta_3= (x_1-x_2)(x_1-x_3)(x_2-x_3)$. (Note that an analog of Lemma~\ref{stab} does not hold here. So we have to consider a non-standard tableau also to generate $I_{l,m,  \lambda}$.) 
Clearly, $\Delta_3 \not \in I_{1,3,\lambda}$, but we can show that $\Delta_3 \in \sqrt{I_{1,3,\lambda}}$ by Lemma~\ref{radical inclusion} below. 
Moreover, the statement corresponding to Lemma~\ref{inclusion} does not hold for  $I_{l,m,\lambda}$. In fact, if $\lambda=(2,2)$ and $\mu=(2,1,1)$, then $\mu \lhd \lambda$, but $I_{1,3,\mu}=(\Delta_3) \not \subset I_{1,3,\lambda}$.  
\end{example}

However, we have the following. 

\begin{lemma}\label{radical inclusion}
For $\lambda, \mu \in [P_{n+l-1}]_{\ge l}$ with $\lambda \unrhd \mu$, we have $\sqrt{I_{l, m, \lambda}} \supset I_{l, m, \mu}$. 
\end{lemma}

\begin{proof}
It suffices to show that $f_{m, T} \in \sqrt{I_{l, m, \lambda}}$ for all $T \in \tab{l, \mu}$.  By Lemma~\ref{inclusion}, there are some $k \in \NN$, $T_1, \ldots, T_k \in \tab{l,\lambda}$ and $g_1, \ldots, g_k \in S$ such that 
$f_T=\sum g_i f_{T_i}$. Multiplying $\Delta_m$ to both sides, we have $$\Delta_m \cdot f_T=\sum g_i \cdot (\Delta_m \cdot f_{T_i}).$$
Since $f_{m,T_i}$ divides $\Delta_m \cdot f_{T_i}$, we have $\Delta_m \cdot f_T \in I_{l,m,\lambda}$. However, since $\Delta_m \cdot f_T$ divides $(f_{m,T})^2$, we have   $(f_{m,T})^2 \in I_{l,m,\lambda}$.
\end{proof}

For a {\it lower} filter $\Fc \subset [P_{n+l-1}]_{\ge l}$, set 
$$G_{l, m, \Fc} := \{ f_{m,T} \mid T \in \tab{l, \lambda}, \lambda \in \Fc \} \quad \text{and} \quad  I_{l,m,\Fc} := (G_{l, m, \Fc})=\sum_{\lambda \in \Fc } I_{l,m,\lambda}.$$
For an {\it upper} filter  $\Fc \subset [P_{n+l-1}]_{\ge l}$ with the lower filter $\Fc^\cmpl:=  [P_{n+l-1}]_{\ge l} \setminus \Fc$, we consider the ideal 
$$J_{l, m, \Fc} := (\Delta_m) \cap J_{l, \Fc} \ (= (\Delta_m) \cap I_{l, \Fc^\cmpl}).$$
Since both $(\Delta_m)$ and $J_{l, \Fc}$ are radical ideals, so is $J_{l, m, \Fc}$. Since  $J_{l,m,\Fc} \subset (\Delta_m)$,  the codimension of $J_{l,m,\Fc}$ for $m \ge 2$ is 1 (unless $\Fc=[P_{n+l-1}]_{\ge l}$, equivalently, $J_{l,m,\Fc}=0$).

\begin{theorem}\label{main2}
Let $\Fc \subsetneq [P_{n+l-1}]_{\ge l}$ be a lower filter, and $\Fc^\cmpl:=  [P_{n+l-1}]_{\ge l} \setminus \Fc$ its complement. Then $G_{l, m, \Fc}$ is a Gr\"obner basis of $J_{l, m, \Fc^\cmpl}$. Hence $J_{l, m, \Fc^\cmpl}=I_{l, m, \Fc}$, and $I_{l, m, \Fc}$ is a radical ideal. 
\end{theorem}


Let us prepare the proof of Theorem~\ref{main2}.

\begin{lemma}
\label{keylemma3}
Let $\Fc \subset [P_{n+l-1}]_{\ge l}$ be an upper filter, and let $f$ be a polynomial in $J_{l,m, \Fc}$
of the form
$$
f= g_d x_n^d  + \cdots +  g_1 x_n +g_0,
$$
where $g_0,\dots,g_d \in K[x_1,\dots,x_{n-1}]$ and $g_d\ne 0$.
If $m <n$, then $g_0,\dots,g_d$ belong to $J_{l, m, \Fc_{d+1}}$. 
\end{lemma}


\begin{proof}
Here we use the same notation as in the proof of Lemma~\ref{keylemma}. 
Take $\bm a=(a_1, \ldots, a_{n-1}) \in K^{n-1}$. Since $f \in (\Delta_m)$, if $a_i =a_j$ for some $1 \le i < j \le m$, then 
$f(\bm a, \alpha)=0$ for all $\alpha \in K$, and hence $g_i(\bm a)=0$ for all $i$. 
It means that each $g_i$ can be divided by $\Delta_m$ in $K[x_1, \ldots, x_{n-1}]$.   So it remains to show that $g_i \in J_{l, \Fc_{d+1}}$, but it follows from Lemma~\ref{keylemma}, since $f \in J_{l, \Fc}$.
\end{proof}

\noindent{\it The proof of Theorem~\ref{main2}.} First, we show that $G_{l, m, \Fc} \subset J_{l, m, \Fc^\cmpl}$. For any $f_{m,T}\in G_{l, m, \Fc}$, it is clear that $f_{m,T}\in (\Delta_m)$, and we have $f_{m,T} \in (f_T) \subset J_{l, \Fc^\cmpl}$ by Theorem~\ref{main1}. Hence $f_{m,T} \in J_{l, m, \Fc^\cmpl}$. 

So it remains to show that, for any $0 \ne f \in J_{l, m, \Fc^\cmpl}$, there is some $f_{m,T} \in G_{l, m, \Fc}$ such that $\ini{f_{m,T}}$ divides $\ini{f}$, but it can be done  by induction on $n-m$ (we fix $m$) in the same way as in the proof of Theorem~\ref{main1}, while we use Lemma~\ref{keylemma3} instead of Lemma~\ref{keylemma}.
\qed 

\bigskip

The following corollary immediately follows from Theorem~\ref{main2}.


\begin{corollary}\label{I_{l,m,lambda}}
For $\lambda \in [P_{n+l-1}]_{\ge l}$, 
$$\bigcup_{\substack{\mu \in [P_{n+l-1}]_{\ge l} \\ \mu \unlhd \lambda}}G_{l,m,\mu}$$
is a  Gr\"obner basis of $\sqrt{I_{l,m, \lambda}}=J_{l,m,\Fc}$, where $\Fc$ is the upper filter $\{\nu \in  [P_{n+l-1}]_{\ge l}   \mid \nu \not \! \! \unlhd \lambda\}$. In particular, 
$$\sqrt{I_{l,m, \lambda}}=\sum_{\substack{\mu \in [P_{n+l-1}]_{\ge l} \\ \mu \unlhd \lambda}}I_{l,m, \mu}.$$
\end{corollary}

\begin{remark}
If $\lambda=(\lambda_1, \ldots, \lambda_p) \in P_{n+l-1}$ is of the form  $\lambda_1=\cdots =\lambda_{p-1}=k-1$ for some $k >l$, then our  $\sqrt{I_{l, m, \lambda}} \ ( =\sum_{\mu \unlhd \lambda} I_{l,m, \mu})$ coincides with the  Li-Li ideal $I_\Yc$ for $\Yc=(Y_1, Y_2, \ldots, Y_{k-1})$ with $Y_1=\{1, 2, \ldots, m\}$, $Y_2=\cdots =Y_l=\{1\}$ and $Y_{l+1}=\cdots =Y_{k-1}=\emptyset$ in the notation of the Introduction.  
\end{remark}

\begin{proposition}
$I_{l,m,\lambda}$ is a radical ideal for $m \le 2$. 
\end{proposition}

\begin{proof}
The case $m=1$ follows from Theorem~\ref{main1}. So we treat the case $m=2$. 
By Theorem~\ref{main2}, it suffices to show that $f_{2,T} \in I_{l,2, \lambda}$ for all $T \in \tab{l, \mu}$ with $\mu \unlhd \lambda$. If the letters 2 and (some copy of) 1 are not in the same column of $T$, then we have $f_{2,T}=(x_1-x_2)f_T$, and if they are in the same column, then  we have $f_{2,T}=f_T$. We first treat the former case. Since $I_{l, \mu} \subset I_{l, \lambda}$ by Lemma~\ref{inclusion}, there are $g_1, \ldots, g_k \in S$ and $T_1, \ldots, T_k \in \tab{l, \lambda}$ such that $f_T=\sum_{i=1}^k g_i f_{T_i}$. Multiplying $(x_1-x_2)$ to the both sides, we have 
$$f_{2,T}=(x_1-x_2)f_T =\sum_{i=1}^k g_i \cdot (x_1-x_2)f_{T_i}.$$
Since $f_{2,T_i}$ divides $(x_1-x_2)f_{T_i}$, we have $f_{2,T} \in I_{l,2, \lambda}$. So the case when 1 and 2 are in the same column (equivalently, $f_{2,T}=f_T$) remains. We may assume that $\lambda$ covers $\mu$, and we want to modify the argument of the proof of  Lemma~\ref{inclusion}, which shows that $I_{l, \mu} \subset I_{l, \lambda}$. In the sequel, we use the same notation as there. 

The crucial case is that $1,2 \in A$ (we may assume that $a_1=1, a_2=2$) and $1 \not \in B$. In fact, in other cases, it is easy to see that  $f_{2,T_i}=f_{T_i}$ for all $i$.  By \eqref{inclusion expansion}, we have 
$$f_T=\sum_{k-k' \le i \le k} (-1)^{i-k+k'} (x_1-x_{a_i})f_{T_i}$$
and $T_i \in \tab{l, \lambda}$ for all $i$. 
For $i \ge 3$,  the letters 1 and 2 stay in the same column of $T_i$, and we have $f_{2,T_i}=f_{T_i}$. So the case $k-k' \ge 3$ is easy, and we may assume that $k-k'=2$. Then, among $T_2, \ldots, T_k$, only $T_2$ does {\it not} have 1 and 2 in the same column. Hence
\begin{eqnarray*}
f_{2,T}=f_T&=& (x_1-x_2) f_{T_2} + 
\sum_{3 \le i \le k} (-1)^i (x_1-x_{a_i})f_{T_i} \\
&=& f_{2, T_2} + 
\sum_{3 \le i \le k} (-1)^i (x_1-x_{a_i})f_{2, T_i}  \in I_{l,2, \lambda}.  
\end{eqnarray*}
\end{proof}

\section*{Acknowledgements}
We would like to thank Professor Hiroshi Teramoto for helpful advice for computer experiments.  We also thank Professor Hidefumi Ohsugi for the quick proof of Lemma~\ref{product of linear forms}. Finally, we are grateful for the anonymous reviewers for their valuable comments.

\end{document}